\documentclass{article}     
\usepackage{amsmath}
\usepackage{amsthm}
\usepackage{amsfonts}
\usepackage{mathrsfs}
\usepackage[T1]{fontenc}
\newcommand{\R}{\mathbb{R}}
\theoremstyle{plain}
\newtheorem{theorem}{Theorem}[section]
\newtheorem{lemma}[theorem]{Lemma}

\newtheorem{proposition}[theorem]{Proposition}

\newcommand\F{{\mathscr F}}

\newcommand\M{{\mathcal M_{T}^{+}}}
\newcommand\dx{{\mathrm d}}

\begin{document}

\title{Supersolutions for a class of semilinear heat equations}



\author{James~C.~Robinson\footnote{Mathematics Institute,
Zeeman Building,
University of Warwick,
Coventry, CV4 7AL, UK, Tel.: +44 (0)24 7652 4657, Fax: +44 (0)871 256 4140, j.c.robinson@warwick.ac.uk} \and     Miko\l{}aj Sier\.z\k{e}ga\footnote{Mathematics Institute,
Zeeman Building,
University of Warwick,
Coventry, CV4 7AL, UK, m.l.sierzega@warwick.ac.uk   }    }


\date{}
%

\maketitle

\begin{abstract}
A semilinear heat equation $u_{t}=\Delta u+f(u)$ with nonnegative initial data in a subset of $L^{1}(\Omega)$ is considered under the assumption that $f$ is nonnegative and nondecreasing and $\Omega\subseteq \R^{n}$. A simple technique for proving existence and regularity based on the existence of supersolutions is presented, then a method of construction of local and global supersolutions is proposed. This approach is applied to the model case $f(s)=s^{p}$, $\phi\in L^{q}(\Omega)$: new sufficient conditions for the existence of local and global classical solutions are derived in the critical and subcritical range of parameters. Some possible generalisations of the method to a broader class of equations are discussed.  
\end{abstract}

\section{Introduction}
In this paper we consider the question of the local existence of solutions to the semilinear heat equation 
\begin{align}\label{eq:main}
&u_{t}-\Delta u=f(u) \quad \mbox{in}\quad \Omega_{T}, \nonumber\\
&u=0 \quad \mbox{on}\quad \partial\Omega\times (0,T), \\
&u(0)=\phi\quad \mbox{in}\quad\Omega,\nonumber
\end{align}
where $\Omega$ is a smooth domain in $\R^{n}$ and $\Omega_{T}=\Omega\times (0,T)$. We assume throughout that the source term $f:[0,\infty)\mapsto [0,\infty)$ is continuous and nondecreasing. The initial data is always assumed to be an element of $L^{1}_{+}(\Omega)$ i.e.\ the set of a.e.\ nonnegative, integrable functions on $\Omega$, or some subset thereof to be specified in the course of the presentation.\par
We are interested in classical solutions, that is functions $u:\Omega_{T}\mapsto \R$ with $u(0)=\phi$ that satisfy the equation pointwise with derivatives understood in their classical sense and continuous up to the boundary $\partial \Omega$. It is known that when $\phi\in L^{\infty}(\Omega)$ then the problem admits a local classical solution, see e.g.\ \cite{SOUPLET}. For unbounded data however statements of such generality are not available and more specific model problems are considered with initial data belonging to some well-undestood Banach or Hilbert space and some particular choice of the source term allowing for analysis. \par
As indicated in the opening paragraph we also impose restrictions on the initial data and the source term. However apart from nonnegativity and integrability we do not impose any additional conditions on the data allowing the existence argument to identify the subset of admissible data. \par
We follow the standard practice of tackling the problem (\ref{eq:main}) indirectly via the associated integral formulation, the so-called variation of constants formula
\begin{equation}\label{eq:VoC}
u(t)=S(t)\phi+\int_{0}^{t}S(t-s)f(u(s))\dx s.
\end{equation}
In this formulation $\{S(t)\}_{t\geq 0}$ represents the heat semigroup associated with the domain $\Omega$ and is defined as 
\begin{equation}\label{eq:ker}
 (S(t)\phi)(x)=\int_{\Omega}K(t,x,y)\phi(y)\dx y\quad \mbox{ for }\quad x\in \overline \Omega,
 \end{equation}
where $K\geq 0$ is the heat kernel for the domain $\Omega$, see e.g. \cite{CAZENAVEBOOK}. In other words the function $v(t)=S(t)\phi$ solves the linear heat equation 
 \begin{align*}
&v_{t}-\Delta v=0 \quad \mbox{in}\quad \Omega, \ t>0,\nonumber\\
&v=0 \quad \mbox{on}\quad \partial\Omega,\ t>0, \\
&v(0)=\phi\quad \mbox{in}\quad\Omega.\nonumber
\end{align*}
Every solution of the integral formulation (\ref{eq:VoC}) that is bounded for $t>0$ must solve (\ref{eq:main}). Hence the route to classical solutions of the original partial differential equation proceeds through analysis of bounded solutions of the variation of constants formula. \par
The existence argument involves the operator
\begin{equation}\label{f}
\F[v](t)=S(t)\phi+\int_{0}^{t}S(t-s)f\left(v(s)\right)\dx s
\end{equation}
which we will use to define a sequence of functions converging to a solution $u=\F[u]$ on $\Omega_{T}$. The method used relies on monotonicity of the operator $\F$ which results from positivity of the heat semigroup and monotonicity of the source term $f$. It should be mentioned that a similar technique was used by Weissler in his investigation of the global solutions to the model problem $f(s)=s^{p}$ in the whole space, see \cite{WEISSLER1981}. \par
The statement of existence of local classical solutions is divided into two parts. The abstract existence proof assumes only that we have a supersolution i.e.\ a function $w$ such that $\F[w]\leq w$. Existence of such an object combined with monotonicity of $\F$ will turn out to give us existence of a solution of the integral formulation. Hence the problem of proving existence of a solution reduces to finding a supersolution which, as we will show in the sequel, is relatively straightforward in many interesting cases. An immediate benefit of the technique is that the regularity of the solution is then inferred from the properties of the supersolution e.g.\ if the supersolution is bounded for positive times then it follows from the standard regularity results that the solution is necessarily classical, see \cite{SOUPLET}. Hence the regularity is ``for free'' and does not have to be obtained by means of bootstrap arguments as e.g.\ in \cite{BREZISCAZENAVE1996,SOUPLET}. \par
In the second part of the argument we provide a practical way of finding suitable supersolutions. We will arrive at certain structural conditions specifying the initial data that allow in certain examples to define a supersolution and thus deduce existence of a solution of (\ref{eq:VoC}). In particular we will see that the central object is the integral 
\begin{equation}
\int_{0}^{t}S(t-s)f\left(S(s)\phi\right)\dx s
\end{equation}
which involves all the parameters of the problem. \par
Next we test our findings on the model problem $f(s)=s^{p}$ with initial data in $L^{q}_{+}(\Omega)$. The problem of existence of local and global classical solutions is well-understood and explicit criteria discriminating well and ill-posed problems are known and presented in literature in the form of inequalities involving critical exponents. For this model problem the critical exponent is given by $q_{c}=\frac{n(p-1)}{2}$ and well-posedness is usually analysed in three regimes: supercritical $q>q_{c}$, critical $q=q_{c}>1$ and subcritical $q<q_{c}$. These relations reflect the balance between the assumed regularity of the initial data, the growth of the source term and the dimension of the domain. \par
The existence and uniqueness results are obtained using Banach's fixed point argument in a suitable metric space of curves in $L^{q}(\Omega)$ satisfying certain norm decay properties, see \cite{BREZISCAZENAVE1996,SOUPLET}. The original proof relying on the Banach contraction theorem appeared in \cite{WEISSLER1979}. In the same paper the author offered also a second method of proof based on the positivity properties of the heat semigroup that is similar to our approach. In the discussion that follows the author compares the two existence theorems to conclude that the positivity based argument recovers the existence result in the supercritical range of parameters, but does not reproduce the results in the critical and subcritical cases.  In particular the positivity method as described in \cite{WEISSLER1979} does not allow one to deduce the existence of local classical solutions for all initial data in $L^{q}(\Omega)$ in the critical case. The proof of this result relies on a delicate application of smoothing estimates in Lebesgue spaces followed by a bootstrap argument, see \cite{BREZISCAZENAVE1996} for details. \par
The particularly appealing property of the positivity argument is that it yields pointwise bounds on the solution in the form of space-time profiles whereas the usual approach offers a rate of decay of Lebesgue norms instead. Moreover, as we will show later on, the positivity based existence theorem is stated in terms of integrability properties of the map $t\mapsto \|S(t)\phi\|_{L^{\infty}(\Omega)}$ rather than critical exponents associated with smoothing estimates in Lebesgue spaces. An assumption of this kind is much more general as it does not presuppose the space of initial data. For example the existence claim in the supercritical case mentioned above requires only that the integral $\int_{0}^{\tau}\|S(s)\phi\|_{L^{\infty}(\Omega)}^{p-1}\dx s$ is bounded for some $\tau>0$ which may be realised irrespective of $\phi$ being in any particular Lebesgue space, whereas the critical exponent condition $n(p-1)/2<q$ is specifically related to the context of Lebesgue spaces. It is therefore desirable to extend the method based on positivity to the critical and subcritical range. This is done in Section \ref{example} where in fact the three ranges of parameters are addressed in a unified approach.  \par
The method also provides sufficient criteria for global existence without any additional effort. As mentioned before, global solutions for the model case $f(s)=s^{p}$ where investigated in \cite{WEISSLER1981}. The global existence result found there relies on the smallness of $\int_{0}^{\infty}\|S(s)\phi\|_{L^{\infty}(\Omega)}^{p-1}\dx s$ in the supercritical case or smallness of $\|\phi\|_{L^{q}(\Omega)}$ in the critical case. Again in the supercritical case the argument uses the positivity of the heat semigroup and provides a space-time profile. However, in the critical case a norm based technique is used instead. Our construction of the supersolutions provides new sufficient criteria for existence of global classical solutions in the critical and subcritical range of parameters.\par
The general construction of supersolutions contained in Proposition \ref{prop:main} is by no means restricted to the case of simple polynomial nonlinearity and may be applied to other source terms of interest. In the last section we discuss some possible extensions.
  
\section{Existence result involving supersolutions}\label{sec:main}
Let $\mathcal M_{T}^{+}$ denote the set of nonnegative, almost everywhere finite, measurable functions on $\Omega_{T}$. For our purposes it suffices to define a solution to be any $u\in \mathcal M_{T}^{+}$ such that $\F[u]=u$ a.e.\ in $\Omega_{T}$. Any $w\in \mathcal M_{T}^{+}$ satisfying $\F[w]\leq w$ (resp. $\F[w]\geq w$) will be called a supersolution (respectively subsolution).\par
The following result serves to reduce the problem of showing existence of a solution to that of showing the existence of a supersolution. 
\begin{theorem}\label{thm:main}
Assume that $f:[0,\infty)\mapsto [0,\infty)$ is continuous, nondecreasing and let $\phi\in L^{1}_{+}(\Omega)$. Then the operator $\F$ admits a solution in $\Omega_{T}$ if and only if it admits a supersolution in $\Omega_{T}$.   
\end{theorem}
\begin{proof}
By definition every solution is at the same time a supersolution and we only need to show that existence of a supersolution implies existence of a solution. \par
Our existence argument relies on positivity and monotonicity properties of the operator $\F$. Take $u,v\in \M$, $u\leq v$ a.e., $u(0)\leq v(0)$ a.e.\ in $\Omega$ such that $\F[u],\F[v]\in \M$. Then 
\[
\F[v](t)-\F[u](t)=S(t)(v(0)-u(0))+\int_{0}^{t}S(t-s)\left(f\big(v(s)\big)-f\big(u(s)\big)\right)\geq 0
\]
follows since by the monotonicity assumption on $f$ we have $f(v(s))-f(u(s))\geq 0$ and the heat semigroup is positivity preserving, see \cite{SOUPLET}. \par
Starting with the supersolution $w$ we can create a sequence $\{w_{k}=\F^{k}[w]\}_{k\geq 0}$, $w_{0}=w$. Now due to monotonicity and positivity of $\F$ we find that 
\[
\F[w]\leq w\Rightarrow  \F^{2}[w]\leq \F[w]\Rightarrow \dots \Rightarrow \F^{k+1}[w]\leq \F^{k}[w]\dots \quad \mbox{ a.e. in }\Omega_{T}. 
\]
Thus over almost every point $(t,x)\in \Omega_{T}$ we have a nonincreasing, nonnegative sequence $\{w_{k}(t,x)\}_{k\geq 0}$ which allows us to define the candidate solution by taking the pointwise limit
\[
w_{\infty}(t,x)=\lim_{k\to\infty}w_{k}(t,x) \quad \mbox{ whenever it exists}.
\]  
\par
It remains to show that $\F[w_{\infty}]=w_{\infty}$ a.e.\ in $\Omega_{T}$. With the integral representation (\ref{eq:ker}) in mind we can apply the monotone convergence theorem to get 
\[
\lim_{k\to\infty}\F[w_{k}](t,x)=S(t)\phi+\int_{0}^{t}\int_{\Omega}K(t-s,x,y)\lim_{k\to\infty}f(w_{k}(s,y))\dx y\,\dx s
\]
a.e.\ in $\Omega_{T}$. 
By continuity
\[
\lim_{k\to\infty}f(w_{k}(s,y))=f\left(\lim_{k\to\infty}w_{k}(s,y)\right)=f(w_{\infty}(s,y))\quad \mbox{ a.e.\ in }\Omega_{T}.
\]
Thus we have continuity of $\F$ at $w_{\infty}$ i.e.\ 
\[
\lim_{k\to\infty}\F[w_{k}]=\F\left[\lim_{k\to\infty}w_{k}\right]=\F[w_{\infty}]\quad \mbox{ a.e.\ in }\Omega_{T}.
\]
Now, due to the construction of the sequence we have $w_{k+1}=\F[w_{k}]$ and so we finally arrive at 
\[
w_{\infty}=\lim_{k\to\infty}w_{k+1}=\lim_{k\to\infty}\F[w_{k}]=\F[w_{\infty}]
\]
a.e.\ in $\Omega_{T}$, which means that $w_{\infty}$ is a solution. 
\end{proof}
 
\section{Identifying supersolutions}
Theorem \ref{thm:main} presupposes existence of a supersolution without indicating how to construct one. In this section we will present a practical way of finding a candidate supersolution. The algorithm will be then tested on a specific well researched model case of a polynomial nonlinearity and compared with known results.\par
The guiding principle is that equation (\ref{eq:VoC}) is a perturbation formula in the sense that the solution is assumed to be given by the linear heat flow $S(t)\phi$ which is then supplemented by a correction term that accounts for the effect of the nonlinearity. Hence we would expect that the contribution of the integral term is initially small in some sense as compared to $S(t)\phi$. The aim of this section is to devise a method of modifying the integral part of the formula (\ref{eq:VoC}) so that the resulting object is a prospective supersolution.  \par
Suppose that there exists a solution $u:\Omega_{T}\mapsto[0,\infty)$, then 
\[
u(t)=S(t)\phi+\int_{0}^{t}S(t-s)f\left(u(s)\right)\dx s \geq S(t)\phi
\] 
and if we now apply the operator $\F$ to both sides repeatedly then due to monotonicity we will find that 
\[
\F^{k}[S(\cdot)\phi](t)\leq u(t)\quad \mbox{ for every }\ k\geq 0
\]
i.e.\ every $\F^{k}[S(\cdot)\phi](t)$ is a subsolution. \par
Our strategy now is to take the first nontrivial subsolution, which is $\F[S(\cdot)\phi](t)$,  and by bounding it from above derive a candidate supersolution. Hence broadly speaking we will replace 
\[
\F[S(\cdot)\phi](t)=S(t)\phi+\int_{0}^{t}S(t-s)f\left(S(s)\phi\right)\dx s
\]
with
\[
w(t)=S(t)\phi+\int_{0}^{t}S(t-s)F(s)\dx s
\]
for some $F\in \M$ such that 
\[
\int_{0}^{t}S(t-s)f\left(S(s)\phi\right)\dx s\leq \int_{0}^{t}S(t-s)F(s)\dx s
\]
hoping that $w$ defines a supersolution on $\Omega_{T}$ for some $T>0$. The optimal way of finding a supersolution depends on $f$ and no one general recipe seems available. Nevertheless there is a simple generic approach which gives satisfactory results.

\begin{proposition}\label{crit:prop}
Suppose there exists an integrable function $h:[0,T]\mapsto [0,\infty)$ and some $\psi \in L_{+}^{1}(\Omega)$ such that for $t\in[0,T]$

\begin{equation}\label{crit:eq2}
f\left(S(t)\phi+S(t)\psi\int_{0}^{t}h(s)\dx s\right)\leq h(t)S(t)\psi,
\end{equation}
then 
\[
w(t)=S(t)\phi+ S(t)\psi\int_{0}^{t}h(s)\dx s
\]
is a supersolution on $\Omega_{T}$. 
\end{proposition}
\begin{proof}
For $t\in [0,T]$ we have 
\begin{align*}
\F[w](t)&= S(t)\phi+\int_{0}^{t}S(t-s)f\left(S(s)\phi+S(s)\psi\int_{0}^{s}h(r)\dx r\right)\dx s\\
&\leq S(t)\phi+\int_{0}^{t}S(t-s) h(s)S(s)\psi\dx s.
\end{align*}
Since $h$ does not depend on the space variables and due to the definition of the heat semigroup we can write
\begin{align*}
\int_{0}^{t}S(t-s)h(s)S(s)\psi\dx s&=\int_{0}^{t}h(s)S(t-s)S(s)\psi\dx s\\
&=\int_{0}^{t}h(s)S(t)\psi\dx s=S(t)\psi\int_{0}^{t}h(s)\dx s
\end{align*}
and so 
\[
\F[w](t)\leq S(t)\phi+S(t)\psi\int_{0}^{t}h(s)\dx s
\]
for $t\in [0,T]$ i.e.\ $\F[w]\leq w$ on $\Omega_{T}$ as required. 
\end{proof}
The next step is to propose a way of defining function $h$. 
\begin{proposition}\label{prop:main}
Let $\phi\in L^{1}_{+}(\Omega)$ and suppose that there exist constants $A>1$, $T>0$ and functions $\psi\in L^{1}_{+}(\Omega)$, $g:[0,\infty)\mapsto [0,\infty)$ and $h:[0,T]\mapsto [0,\infty)$ such that  
\begin{equation}\label{cond1}
S(t)\phi\leq g\left(S(t)\psi\right)\mbox{ and }\quad 1+\left\|\frac{S(t)\psi}{g(S(t)\psi)}\right\|_{L^{\infty}(\Omega)} \int_{0}^{t}h(s)\dx s\leq A
\end{equation}
for $t\in [0,T]$, where 
\[
h(t)=\left\|\frac{f\left(A g\left(S(t)\psi\right)\right)}{S(t)\psi}\right\|_{L^{\infty}(\Omega)}.
\]
Then $w(t)=S(t)\phi+S(t)\psi \int_{0}^{t}h(s)\dx s$ is a supersolution for $t\in [0,T]$.
\end{proposition}
\begin{proof}
The conditions listed above serve to ensure that the inequality (\ref{crit:eq2}) holds. We have 
\begin{align*}
f\left(S(t)\phi+S(t)\psi\int_{0}^{t}h(s)\dx s\right)&\leq f\left(g\left(S(t)\psi\right)+S(t)\psi\int_{0}^{t}h(s)\dx s\right)\\
&\leq f\left(g\left(S(t)\psi\right)\left(1+\left\|\frac{S(t)\psi}{g(S(t)\psi)}\right\|_{L^{\infty}(\Omega)}\int_{0}^{t}h(s)\dx s\right)\right)\\
&\leq f\left(A g\left(S(t)\psi\right)\right)\leq h(t)S (t)\psi.
\end{align*}
\end{proof}
The function $g$ used above is intended to represent an operation that allows us to extract the presupposed regularity of the initial condition. In the next section we will see how one may use the Jensen inequality to take advantage of the $L^{q}$- integrability of $\phi$. More precisely Lemma \ref{lem:jen} informs us that for $r\geq 1$ and $\phi$ nonnegative we have $(S(t)\phi)^{r}\leq S(t)\phi^{r}$. Thus for $\phi\in L^{q}_{+}(\Omega)$ we can write 
\[
S(t)\phi=\left(\left(S(t)\phi\right)^{q}\right)^{\frac{1}{q}}\leq \left(S(t)\phi^{q}\right)^{\frac{1}{q}}
\]
so that the correspondence to the above is given by $\psi=\phi^{q}$ and $g(s)=s^{1/q}$. \par
Observe that in the results above the applicability of a given supersolution is valid as long as the inequalities (\ref{crit:eq2}) and (\ref{cond1}) hold. Hence if we consider the set of initial conditions for which $T=\infty$ then our results turn into the statements of global existence of classical solutions for small data. In this case the global supersolution $w(t)=S(t)\phi+S(t)\psi \int_{0}^{t}h(s)\dx s$ may be for example used to obtain a pointwise bound on the asymptotic profile of the solution. 
Note however that the smallness of data is understood in the sense of said inequalities and  may differ from the statements found in the literature where the smallness is sometimes defined in terms of the norm of the initial condition in some relevant functional setting, cf. \cite{SOUPLET,WEISSLER1981}.   
\section{Example - $f(u)=u^{p}$}\label{example}
In this section we restrict our attention to the case $\phi\in L^{q}_{+}(\Omega)$, $q\geq 1$. For $p>1$ consider the following model problem: 
\begin{align}\label{model}
&u_{t}=\Delta u+u^{p}\quad \mbox{ in }\quad\Omega,\ t>0\nonumber,\\
&u=0 \quad \mbox{ on }\quad\partial \Omega,\\
&u(0)=\phi \quad\mbox{ in }\quad\Omega.\nonumber
\end{align}
Existence and uniqueness results for this problem were established by Weissler \cite{WEISSLER1980} and then augmented by Brezis and Cazenave \cite{BREZISCAZENAVE1996}. For comparison purposes we present an abridged version of their findings concentrating only on the existence of classical solutions for nonnegative data.  
\begin{theorem}\label{thm:weiss}
Let $\phi\in L^{q}_{+}(\Omega)$, $q\geq 1$, and suppose that one of the following cases holds:
\begin{enumerate}
\item supercritical: $q>n(p-1)/2$,
\item critical: $q=n(p-1)/2>1$, 
\item subcritical: $n(p-1)/2 p<q<n(p-1)/2$ and 
\begin{equation}\label{conv}
\lim_{t\to 0}\left\|t^{\alpha}S(t)\phi\right\|_{L^{p q}(\Omega)}=0, \quad \mbox{where }\alpha=\frac{1}{p-1}-\frac{n}{2 p q},  
\end{equation}
\end{enumerate}
then there exists $T=T(\phi)>0$ such that the problem (\ref{model}) has a local classical solution on $\Omega_{T}$.  
\end{theorem}
As mentioned in the introduction the method of proof relies on the contraction argument in a carefully chosen space of curves followed by the bootstrap argument for regularity. In \cite{WEISSLER1980} the supercritical case was also resolved using positivity of the heat semigroup but the critical range was not covered.  Here we present a simple way of extending the positivity-based method to include critical and subcritical ranges. \par 
Before we present the full result let us see how the local existence and regularity in the supercritical range $n(p-1)/2<q$ can be recovered in a short direct computation. To do this it suffices to use the standard $L^{q}-L^{r}$ smoothing estimate \cite{CAZENAVEBOOK}: 
\begin{equation}\label{stsmooth}
\|S(t)\phi\|_{L^{r}(\Omega)}\leq t^{-\frac{n}{2 }\left(\frac{1}{q}-\frac{1}{r}\right)}\|\phi\|_{L^{q}(\Omega)}, 
\end{equation}
where $1\leq q\leq r\leq \infty$. \par
Let $A>1$ be a constant. We will show that in this case $A S(t)\phi$ is a supersolution for (\ref{model}) on some sufficiently small time interval. We have 
\begin{align*}
\F[A S(\cdot)\phi](t)&\leq S(t)\phi+\int_{0}^{t}S(t-s)\left(A^{p}\|S(s)\phi\|_{L^{\infty}(\Omega)}^{p-1}S(s)\phi\right)\dx s\\
&=S(t)\phi\left(1+A^{p}\int_{0}^{t}\|S(s)\phi\|_{L^{\infty}(\Omega)}^{p-1}\dx s\right)\leq A S(t)\phi,
\end{align*}
whenever $\int_{0}^{t}\|S(s)\phi\|_{L^{\infty}(\Omega)}^{p-1}\dx s$ is finite and $t$ is small enough. In the considered supercritical range the smoothing estimate yields 
\[
\int_{0}^{t}\|S(s)\phi\|_{L^{\infty}(\Omega)}^{p-1}\dx s\leq \|\phi\|_{L^{q}(\Omega)}^{p-1}t^{1-\frac{n(p-1)}{2 q}}
\]
and for every $\phi\in L^{q}_{+}(\Omega)$ there exists a time $T=T\left(\|\phi\|_{L^{q}(\Omega)}\right)$ such that
\begin{equation}\label{eq:smth}
1+A^{p}\|\phi\|_{L^{q}(\Omega)}^{p-1}t^{1-\frac{n(p-1)}{2 q}}\leq A
\end{equation}
for $t\in [0,T]$. \par
It is worth mentioning that in order to obtain the global supersolution for small data the standard smoothing estimate is not enough. For every $A>1$ and $\|\phi\|_{L^{q}(\Omega)}$ there will be a time such that (\ref{eq:smth}) ceases to hold. However the global validity of the supersolution depends on the smallness of the $\int_{0}^{\infty}\|S(s)\phi\|_{L^{\infty}(\Omega)}^{p-1}\dx s$ rather than the bound obtained with the smoothing estimate. The explanation of the discrepancy lies in the fact that the smoothing estimate works well for small times but is too crude for large times. When $t$ is large one should use the bound involving exponential decay, see \cite{SOUPLET} p. 441. We will now leave the discussion of global supersolutions and come back to it after the more general result is presented. \par   
In order to include critical and subcritical cases a slightly subtler choice of the supersolutions is required. 
We need two additional results. The first one is a particular case of the standard smoothing estimates, see \cite{BREZISCAZENAVE1996} for the full statement and the proof. 

\begin{lemma}\label{lem:smooth}
Given $\phi\in L^{q}(\Omega)$, $1\leq  q<\infty$, then for every $K>0$ there exits a time $\tau=\tau(\phi,K)>0$ such that 
\begin{equation}\label{smooth}
t^{\frac{n}{2 q}}\|S(t)\phi\|_{L^{\infty}(\Omega)}\leq K
\end{equation}
for $t\in [0,\tau]$. 
\end{lemma}
The second result involves the interplay between convex functions and the heat semigroup as found in \cite{WEISSLER1979,WEISSLER1986}. 
\begin{lemma}\label{lem:jen}
Let $\phi\geq 0$ be a measurable function on $\Omega$ and $r\geq 1$, then 
\begin{equation}\label{eq:jen}
(S(t)\phi)^{r}\leq S(t)\phi^{r}.
\end{equation}  
\end{lemma}
The main result of this section follows from a combination of the Proposition \ref{prop:main} and the above mentioned lemmas. 
\begin{theorem}\label{prop:lq}
If 
\begin{equation}\label{condp}
\left\|S(t)\phi^{q}\right\|_{L^{\infty}(\Omega)}^{\frac{q-1}{q}}\int_{0}^{t}\left\|S(s)\phi^{q}\right\|_{L^{\infty}(\Omega)}^{\frac{p-q}{q}}\dx s\leq C_{p}
\end{equation}
for $t\in [0,T]$, $0<T\leq \infty$, where $C_{p}=\frac{(p-1)^{p-1}}{p^{p}}$, then problem (\ref{model}) has a classical solution on $\Omega_{T}$. 
\end{theorem}
\begin{proof}
We simply identify the elements of the Proposition \ref{prop:main} in the current setting. Lemma \ref{lem:jen} provides us with the explicit choice of $g(s)=s^{1/q}$ and $\psi=\phi^{q}$ so that the supersolution is given by 
\[
w(t)=S(t)\phi+A^{p}S(t)\phi^{q}\int_{0}^{t}\left\|S(s)\phi^{q}\right\|_{L^{\infty}(\Omega)}^{\frac{p-q}{q}}\dx s
\]
and the condition (\ref{cond1}) reads 
\[
\|S(t)\phi^{q}\|_{L^{\infty}(\Omega)}^{\frac{q-1}{q}}\int_{0}^{t}\|S(s)\phi^{q}\|_{L^{\infty}(\Omega)}^{\frac{p-q}{q}}\dx s\leq \frac{A-1}{A^{p}}
\]
for some $A>1$. Now, the right hand side attains its maximum on $(1,\infty)$ equal to $\frac{(p-1)^{p-1}}{p^{p}}$ yielding the desired result.  
\end{proof}
The supercritical case is immediately recovered once we set $q=1$ in (\ref{condp}). Let us now turn to the critical case $n(p-1)/2 =q>1$. Due to the smoothing estimate (\ref{smooth}) we have 
\begin{align*}
\left\|S(t)\phi^{q}\right\|_{L^{\infty}(\Omega)}^{\frac{q-1}{q}}\int_{0}^{t}\left\|S(s)\phi^{q}\right\|_{L^{\infty}(\Omega)}^{\frac{p-q}{q}}\dx s&\leq K^{\frac{p-1}{q}} t^{-\frac{n}{2 q}(q-1)}\int_{0}^{t}s^{-\frac{n}{2 q}(p-q)}\dx s
\end{align*}
and 
\begin{equation}\label{critical}
\frac{n}{2 q}(p-q)=\frac{p-q}{p-1}<1
\end{equation}
since $q>1$. Thus we can continue the estimate to get  
\begin{equation}\label{calc}
K^{\frac{p-1}{q}}t^{-\frac{n}{2 q}(q-1)}\int_{0}^{t}s^{-\frac{n}{2 q}(p-q)}\dx s\leq \frac{p-1}{q-1}K^{\frac{p-1}{q}} \leq C_{p},
\end{equation}
where the last inequality follows from the Lemma \ref{lem:smooth} for times small enough. Hence Theorem \ref{prop:lq} immediately implies Theorem \ref{thm:weiss} with the additional advantage of the space-time profile given by the supersolution.\par
In the subcritical range $n(p-1)/2>q$ it is known that the problem (\ref{model}) is ill-posed in $L^{q}(\Omega)$ in the sense that for any $\phi\in L^{q}(\Omega)$ one can find a sequence $\phi_{n}$ of nonnegative smooth functions convergent to $\phi$ in $L^{q}(\Omega)$ but such that blow-up times of the associated solutions converge to zero, see \cite{ARRIETABERNAL2004}. Thus in this range not every initial condition gives rise to a solution. It is therefore expected that the calculation (\ref{calc}) cannot be extended beyond the critical combination of parameters. Nevertheless the condition (\ref{condp}) does not preclude initial data in the subcritical range. It is clear however that the decay of $\|S(t)\phi\|_{L^{\infty}(\Omega)}$ as described by the smoothing estimate for a general element of $L^{q}(\Omega)$ is not enough to ensure that the inequality in question is satisfied. Thus we are led to consider a subset of initial data characterised by a faster decay of the supremum norm under the action of the heat semigroup. The following result will enable us to relate our findings to the third option of the Theorem \ref{thm:weiss}. 
\begin{proposition}\label{suff}
Let $q>1$ and suppose that $\phi\in L^{q}_{+}(\Omega)$ is such that there exists $\tau=\tau(\phi)$ such that 
\[
t^{\frac{1}{p-1}}\left\|S(t)\phi^{q}\right\|_{L^{\infty}(\Omega)}^{\frac{1}{q}}\leq C_{p}^{\frac{1}{p-1}},
\]
then problem (\ref{model}) admits a local solution. 
\end{proposition}
\begin{proof}
We have 
\[
\left\|S(t)\phi^{q}\right\|_{L^{\infty}(\Omega)}^{\frac{q-1}{q}}\int_{0}^{t}\left\|S(s)\phi^{q}\right\|_{L^{\infty}(\Omega)}^{\frac{p-q}{q}}\dx s\leq C_{p}t^{-\frac{q-1}{p-1}}\int_{0}^{t}s^\frac{p-q}{p-1}\dx s=C_{p}
\]
as required by the condition (\ref{condp}). 
\end{proof}
Observe that for a general element $\phi\in L^{q}(\Omega)$ the Jensen inequality (\ref{eq:jen}) yields 
\[
\|S(t)\phi\|_{L^{p q}(\Omega)}\leq \left\|S(t)\phi^{q}\right\|_{L^{ p }(\Omega)}^{\frac{1}{q}}
\] 
and at the same time the smoothing estimate (\ref{stsmooth}) implies that both norms share the same decay rate $t^{-\frac{n(p-1)}{2 p q}}$. If we then assume that the initial condition $\phi\in L^{q}_{+}(\Omega)$ is such that 
\begin{equation}\label{uff}
t^{\alpha}\left\|S(t)\phi^{q}\right\|_{L^{ p }(\Omega)}^{\frac{1}{q}}\rightarrow 0\quad \mbox{ as }\quad t\to 0
\end{equation}
in place of assumption (\ref{conv}), then we have 
\begin{align*}
\left\|S(2 t)\phi^{q}\right\|_{L^{\infty}(\Omega)}^{\frac{1}{q}}&=\left\|S\left(t\right)S(t)\phi^{q}\right\|_{L^{\infty}(\Omega)}^{\frac{1}{q}}\leq  t^{-\frac{n}{2 p q }}\|S(t)\phi^{q}\|_{L^{p }(\Omega)}^{\frac{1}{q}}
\end{align*}
and so
\[
t^{\frac{1}{p-1}}\|S(t)\phi^{q}\|_{L^{p }(\Omega)}^{\frac{1}{q}}\leq 2^{\alpha}\left(\frac{t}{2}\right)^{\alpha}\left\|S(t/2)\right\|_{L^{p}(\Omega)}^{\frac{1}{q}}\rightarrow 0 \quad \mbox{ as }\quad t\to 0.
\]
Thus Theorem \ref{thm:weiss} and Proposition \ref{suff} agree for initial data satisfying condition (\ref{uff}). It would be interesting to determine whether there are initial conditions that comply with the requirements of Theorem \ref{thm:weiss} but not of Proposition \ref{suff}. \par 
Interestingly the sufficient condition given in the Proposition \ref{suff} bears a striking resemblance to the necessary condition that any nonnegative integral solution of (\ref{model}) must satisfy, see \cite{WEISSLER1986}. In the article quoted we find the following result. 
\begin{proposition}
Suppose that $u\in \M$ is an integral solution of the equation (\ref{model}) in the sense of the variation of constants formula. Then necessarily 
\[
t^{\frac{1}{p-1}}\|S(t)\phi\|_{L^{\infty}(\Omega)}\leq \left(\frac{1}{p-1}\right)^{\frac{1}{p-1}}
\]
for $t\in (0,T]$ and $x\in \Omega$. 
\end{proposition}   
This intriguing formal similarity between the sufficient and the necessary conditions requires more research. \par
Let us conclude this section by restating the existence result in a unified form. The inequality (\ref{condp}) gives a well-posedness condition of a more general form that the one involving critical exponents. To identify the admissible data we can define the operator 
\[
P(\phi,T)=\sup_{t\in [0,T]}\left(\left\|S(t)\phi^{q}\right\|_{L^{\infty}(\Omega)}^{\frac{q-1}{q}}\int_{0}^{t}\left\|S(s)\phi^{q}\right\|_{L^{\infty}(\Omega)}^{\frac{p-q}{q}}\dx s\right)
\]
so that we can say that problem (\ref{model}) has classical solution for every initial condition in the set 
\[
X=\left\{\phi\in L^{q}_{+}(\Omega)\ :\ P(\phi,T)\leq C_{p}\mbox{ for some }T\in (0,\infty]\right\},
\]
where the global supersolutions are guaranteed for all $\phi$ for which $P(\phi,\infty)\leq C_{p}$. \par 
With this notation we can rephrase our findings and their correspondence to the existing results. In the supercritical and critical range we find that $L^{q}_{+}(\Omega)= X$. Indeed, the calculation above states that for this choice of parameters every $\phi\in L^{q}_{+}(\Omega)$ satisfies $P(\phi,T)\leq C_{p}$ for sufficiently small $T$ so that $L^{q}_{+}(\Omega)\subseteq X$. On the other hand $X\subseteq L^{q}_{+}(\Omega)$ by definition. In the case $n(p-1)/2q>1$ instead of equality we have a proper inclusion $X\subset L^{q}_{+}(\Omega)$ which reflects the fact that not every initial condition in $L^{q}_{+}(\Omega)$ gives rise to a solution.  
\section{Final comments}
The method of supersolutions shares some features with Banach's fixed point argument which is a standard technique in the field of well-posedness of PDEs. There are however essential differences which may render the method described here easier to use. Observe that there is no need for constructing a particular space for the method to work. The fixed point argument requires a contraction space, i.e.\  a set of curves that when equipped with some appropriate metric becomes a complete metric space and such that the operator $\F$ restricted to this space is a strict contraction. Finding an appropriate choice of the contraction space may prove difficult in many cases and may require great intuition and skill in manipulating inequalities. This point is illustrated by the construction of the contraction space for the model problem (\ref{model}) in \cite{BREZISCAZENAVE1996,WEISSLER1979}. The difficulties arise partly because the contraction argument yields not only existence but also uniqueness and so the contraction space needs to accommodate properties necessary for proving both. The method described in this paper does not address the issue of uniqueness, which we leave for  a dedicated discussion elsewhere, but at the same time asserts the existence of solutions under arguably milder conditions than those usually required in fixed point arguments. \par
One of the most important features of the technique is that the initial data $\phi$ is, to begin with, only required to be nonnegative and integrable and it is the procedure itself that imposes further restrictions in the course of estimates. Hence, it is the source term $f$ that defines the space of admissible initial data rather than the data restricting the choice of the source term. In consequence we may find that the given problem possesses local solutions for initial data that forms a set that is not necessarily a linear space. Such a situation should in general be expected when the equation at hand is a nonlinear one, nevertheless at the time of constructing the contraction space one is usually interested in well-posedness in some ``well-behaved'' space like a Hilbert or a Banach space. \par
The method yields pointwise bounds in the form of space-time profiles which may prove useful in investigations of such topics as decay rate, blow-up rate, blow-up profiles, blow-up sets and asymptotic behaviour in instances when global supersolutions are available.  \par
The method affords many possible generalisations. All the results hold if we relax the requirement that the domain $\Omega$ is smooth and demand instead that it satisfies an exterior cone condition, see \cite{SOUPLET} p. 440 and references therein. In the analysis of the integral equation (\ref{eq:VoC}) the heat semigroup may be replaced by a more general one as long as the requirements of the theorems are satisfied, namely that the family $\left\{S(t)\right\}_{t\geq 0}$ is a positivity preserving semigroup on $L^{1}(\Omega)$ such that $S(t)\phi\in L^{\infty}(\Omega)$ for any $\phi\in L^{1}(\Omega)$. The restriction that $f$ should be nondecreasing may be weakened since the necessary ingredient of the existence proof is the order preserving property of the operator $\F$ restricted to the sequence $\{w_{k}\}_{k\geq 0}$ which is guaranteed if $f$ is nondecreasing but may be realised without this assumption. As far as the existence proof is concerned the requirement that $f$ be continuous may be weakened since a nondecreasing function is almost everywhere continuous and the result is meant to hold almost everywhere. \par
In the example of the polynomial nonlinearity (\ref{model}) the key ingredient that allowed us to extend the monotonicity argument to cover the critical and subcritical cases was the application of the Jensen inequality (\ref{eq:jen}). Thanks to this inequality it was possible to make better use of the assumed integrability of the initial data. Note that the general construction of supersolutions suggested in Proposition \ref{prop:main} does not specify how the extra regularity is obtained. The Jensen inequality is particularly apt in the context of Lebesgue spaces. It would be interesting to see how to extract the required information in other functional contexts e.g.\  Sobolev spaces or Orlicz spaces. \par
In the subcritical range the results obtained with the supersolution method resemble those described in \cite{WEISSLER1979,WEISSLER1986}. At this stage however the similarity is formal and more research is needed in order properly to relate the two approaches.\par
Lastly, it would be desirable to relax the nonnegativity assumption on the initial data and to guarantee uniqueness. These issues require a separate treatment and will be discussed elsewhere.

\section*{Acknowledgments}
Authors would like to thank Dr Alejandro Vidal-Lopez for many stimulating conversations and useful remarks. \\
This research was supported by EPSRC, grant No. EP/G007470/1.

\bibliography{/Users/miko/Stuff/bibliography}
\bibliographystyle{plain}

\end{document}